\documentclass[preprint]{elsarticle}

\usepackage{latexsym,amsthm,amssymb,amsmath,mathrsfs,marvosym}
\usepackage{mathtools}
\usepackage[english]{babel}
\usepackage{url}
\usepackage[pdftex]{color}
\usepackage{kbordermatrix}

\newcommand{\F}{\mathbb F}
\newcommand{\FF}{\mathbb F}

\newcommand{\CC}{\mathbb C}
\newcommand{\C}{\mathbb C}
\newcommand{\RR}{\mathbb R}

\newcommand{\Rc}{\mathcal R}

\newcommand{\FI}{\mathfrak I}
\newcommand{\FJ}{\mathfrak J}

\newcommand{\dps}{\displaystyle}

\def\rank{\mathop{\rm rank}\nolimits}

\def\adj{\mathop{\rm Adj}\nolimits}

\def\Adj{\mathop{\rm Adj}\nolimits}
\def\tr{\mathop{\rm tr}\nolimits}

\newcommand{\be}{\begin{equation}}
\newcommand{\ee}{\end{equation}}

\newcommand{\wt}{\widetilde}

\newcommand{\la}{\lambda}

\newtheorem{theorem}{Theorem}[section]

\newtheorem{proposition}[theorem]{Proposition}
\newtheorem{corollary}[theorem]{Corollary}
\newtheorem{lemma}[theorem]{Lemma}
\newtheorem{rem}[theorem]{Remark}

\title{On a formula of Thompson and McEnteggert for the adjugate matrix}
\fntext[fnK]{Partially supported by the Centre for Mathematics of the University of Coimbra - UIDB/00324/2020,
funded by the Portuguese Government through FCT/MCTES. }
\author[kc]{Kenier Castillo\fnref{fnK}}
\ead{kenier@mat.uc.pt}

\fntext[fnI]{Partially supported by ``Ministerio de Econom\'ia, Industria
y Competitividad (MINECO)''
of Spain and ``Fondo Europeo de Desarrollo Regional (FEDER)'' of EU through grants MTM2017-83624-P
and MTM2017-90682-REDT,  and by UPV/EHU through grant GIU16/42.}
\author[iz]{Ion Zaballa\corref{cor}\fnref{fnI}}
\ead{ion.zaballa@ehu.es}
\address[kc]{
CMUC, Department of Mathematics, 
University of Coimbra,  
3001-501 Coimbra, Portugal.
}
\address[iz]{
Departamento de Matem\'atica Aplicada y EIO.
Universidad del Pa\'{\i}s Vasco (UPV/EHU).
Apdo. 644.
48080 Bilbao. Spain.
}

\begin{document}
\begin{abstract}
For an eigenvalue $\la_0$ of a Hermitian matrix $A$, the formula of Thompson and McEnteggert
gives an explicit expression of the adjugate of $\la_0 I-A$, $\Adj(\la_0I-A)$, in terms of eigenvectors of $A$ for
$\la_0$ and all its eigenvalues. 
In this paper Thompson-McEnteggert's formula is
generalized to include any matrix with entries in an arbitrary field. In addition,  for any nonsingular
matrix $A$, a formula for the elementary divisors of $\Adj(A)$ is provided in terms of those of $A$.
Finally, a generalization of the eigenvalue-eigenvector identity and two applications of the
Thompson-McEnteggert's formula are presented.
\end{abstract}


\maketitle

\section{Introduction}
Let $\Rc$ be a commutative ring with identity.
Following \cite[Ch. 30]{Gode05}, for a polynomial $p(\la))=\sum_{k=0}^n p_k\la^k\in\Rc[\la]$
its derivative is $p'(\la)=\sum_{k=1}^n kp_k\la^{k-1}$. Recall that if $X\in\Rc^{n\times n}$
is a square matrix of order $n$ with entries in $\Rc$ and $M_{ij}(X)$ is the minor obtained from
$X$ by deleting the
$i$th row and $j$th column then the \textit{adjugate} of $X$, $\Adj(X)$, is the matrix whose $(i,j)$ entry is
$(-1)^{i+j} M_{ji}(X)$; that is, $\Adj(X)=\begin{bmatrix} (-1)^{i+j} M_{ji}(X)\end{bmatrix}_{1\leq i,j\leq n}$.

Formula \eqref{eq.TM1} below, from now on TM formula, was proved, with $w=v$ and
the normalization $w^\ast v=1$,  for a Hermitian matrix ${A} \in \CC^{n \times n}$ by Thompson
and McEnteggert  (see  \cite[pp. 212-213]{TM68}).
Inspection of the proof shows that the  formula also holds for normal matrices over $\CC$
(see \cite{S79}). With the same arguments we can go further. Recently, Denton, Parke, Tao,
and Zhang pointed out that the TM formula has an extension to a non-normal matrix
${A}\in\RR^{n \times n}$, so long as it is diagonalizable (see \cite[Rem. 4]{BPTZ20}). Even more,
as shown in Remark 5 of \cite{BPTZ20} it holds for matrices over commutative rings (see \cite{Gr19}
for an informal proof).  A more detailed proof of  this result will be given in Section \ref{sec.TM}. 
However, for matrices  over fields (or over integral domains) with repeated
eigenvalues, \eqref{eq.TM1} does not provide meaningful information (see Remark \ref{rem.1}).
We will exhibit in Section \ref{sec.TM} a generalization of the TM formula which holds for
matrices over arbitrary fields with repeated eigenvalues. This new TM
formula will be used to generalize the so-called \textit{eigenvector-eigenvalue identity}
(see \eqref{eq.ev-ev0}) for non-diagonalizable matrices over arbitrary fields.
In addition we will provide a complete characterization of the similarity invariants of $\Adj(A)$
in terms of those of $A$, generalizing a result about the eigenvalues and the minimal
polynomial in \cite{HiUn85}. Then in Section \ref{sec.cons} two additional consequences of the TM
formula will be analysed.

\section{ The TM formula and its generalization}\label{sec.TM}

Let $A\in\Rc^{n\times n}$ be a square matrix of order $n$ with entries in $\Rc$.
An element $\la_0\in\Rc$ is said to be an  \textit{eigenvalue} of $A$ if $Ax=\la_0 x$
for some nonzero vector $x\in\Rc^{n\times 1}$(\cite[Def. 17.1]{Brown93}).
This vector is said to be a \textit{right eigenvector} of $A$ for
(or associated with) $\la_0$. The \textit{left eigenvectors} of $A$ for $\la_0$ are the right eigenvectors for $\la_0$
of $A^T$, the transpose of $A$, or, if $ \Rc=\CC$ is the field of complex numbers, of $A^\ast$, the conjugate
transpose of $A$. That is to say, $y\in\Rc^{n\times 1}$ is a left
eigenvector  of $A$ for $\la_0$ if $y^TA=\la_0 y^T$ (or $y^\ast A=\la_0 y^\ast$ if $\Rc=\CC$).
The characteristic polynomial of $A$ is $p_A(\la)=\det(\la I_n-A)$ and $\la_0$ is an eigenvalue of $A$ if
and only if $p_A(\la_0)\in Z(\Rc)$, where $Z(\Rc)$ is the set of zero divisors of $\Rc$ (\cite[Lem. 17.2]{Brown93}).

The following result, in a slightly different form, was proved by  D. Grinberg in \cite{Gr19}.

\begin{theorem}\label{thm.1}
Let $A\in\Rc^{n\times n}$ and let $\la_0\in\Rc$ be an eigenvalue of $A$. Let
$v,\; w\in\Rc^{n\times 1}$ be a right and a left eigenvector, respectively, of $A$
for $\la_0$ . Then
\be\label{eq.TM1}
w^Tv \Adj(\la_0 I_n-A)=p'_A(\la_0)vw^T.
\ee
\end{theorem}

The proof in  \cite{Gr19} is based on the Lemma \ref{lem.2}  below which is
interesting in its own right. According to McCoy's theorem (\cite[Th. 5.3]{Brown93}) there is a non-zero
vector $x\in\Rc^{n\times 1}$ such that $Ax=0$ if and only if  $\rm{rk}(A)<n$, where $\rm{rk}(A)$ is the
(McCoy) rank of $A$ (\cite[Def. 4.10]{Brown93}). In other words, $0$ is an eigenvalue of $A$ if and only
if $\rm{rk}(A)<n$. Note that $\rm{rk}(A)=\rm{rk}(A^T)$.

\begin{lemma}\label{lem.2}
Let $A\in\Rc^{n\times n}$ be a matrix  such that $\rm{rk}(A)<n$ and let $w\in\Rc^{n\times 1}$
be a left eigenvector of $A$ for the eigenvalue $0$. For $j=1,\ldots, n$, let $(\Adj A)_j$ be
the $j$th column of $\Adj(A)$.  Then, for all $i,j=1,\ldots,n$,
\be\label{eq.5}
w_i(\adj A)_j=w_j (\adj A)_i,
\ee
where  $w=\begin{bmatrix} w_1& w_2 &\cdots & w_n\end{bmatrix}^T$.
\end{lemma}

This is Lemma 3 of \cite{Gr19}. The author himself considers the proof  to be informal.
So a detailed proof of Lemma \ref{lem.2}, following Grinberg's  ideas\footnote{Grinberg's permission
was granted to include the proofs of  this Lemma and Theorem \ref{thm.1}},
is given next for reader's convenience. 

\begin{proof}[Proof of Lemma \ref{lem.2}]
Let us  take $i,j\in\{1,\ldots, n\}$ and assume that $i\neq j$; otherwise, there is nothing to prove.
We assume also, without lost of generality, that $i<j$.
Let $w=\begin{bmatrix} w_1 & w_2& \cdots & w_n\end{bmatrix}^T$ and,
for $k=1,\ldots, n$, let $a_k$ be the $k$th row of $A$. Define $B\in\Rc^{n\times n}$
to be the matrix whose $k$th row, $b_k$, is equal to $a_k$ if $k\neq i,j$ and $b_k=w_k a_k$
if $k=i,j$. A simple computation shows that $w_i(\Adj A)_j=(\Adj B)_j$ and
$w_j(\Adj A)_i=(\Adj B)_i$. We claim that $(\Adj B)_j$=$(\Adj B)_i$. This would prove the lemma.

It follows from $w^TA=0$ that $\sum_{k=1}^n w_k a_k=0$ and so
\be\label{eq.bimbj}
b_i+b_j=-\sum_{k=1,k\neq,i,j}^n w_k b_k.
\ee
Let
\[
P=\kbordermatrix{
&&&&i&&&j&&& &\\
&1&&&&&&&& &\\
& &\ddots &&&&&&&&&\\
&&&1&&&&&&&&\\
i&-w_1&\cdots&-w_{i-1}&-1&-w_{i+1}&\cdots&-w_{j-1}&0&-w_{j+1}&\cdots&-w_n\\
&&&&&1&&&&&&\\
&&&&&&\ddots&&&&&\\
&&&&&&&1&&&&\\
j&-w_1&\cdots&-w_{i-1}&0&-w_{i+1}&\cdots&-w_{j-1}&-1&-w_{j+1}&\cdots&-w_n\\
&&&&&&&&&1&&\\
&&&&&&&&&&\ddots&\\
&&&&&&&&&&&1\\
}.
\]
This matrix is invertible in $\Rc$ (its determinant is $1$) and by \eqref{eq.bimbj},
\[
\wt{B}=PB=\left[\begin{array}{ccccccccccc}
b_1^T & \cdots &b_{i-1}^T & b_j^T & b_{i+1}^T & \cdots & b_{j-1}^T &
b_i^T & b_{j+1}^T&\cdots &b_n^T\end{array}\right]^T.
\]
Then, $\Adj(\wt{B})=\Adj(B)\Adj(P)$ and, since $P$ is invertible, $\Adj(P)=(\det P) P^{-1}=P^{-1}$. Hence
$\Adj(B)=\Adj(\wt{B})P$ and for $k=1,\ldots, n$
\[
(\adj B)_{ki}=\sum_{\ell=1}^n (\Adj\wt{B})_{k\ell}P_{\ell i}.
\]
But in the $i$th column of $P$ the only nonzero entry is $-1$ in position $(i,i)$. Therefore,
$(\adj B)_{ki}=-(\Adj\wt{B})_{ki}$. Now, taking into account that $\wt{B}$ is the matrix
$B$ with rows $i$th and $j$th interchanged and recalling that $M_{ij}(X)$ is the minor of
$X$ obtained by deleting the $i$th row and $j$th column of $X$, we get
\[
\begin{array}{rcl}
(\adj B)_{ki}&=&-(\Adj\wt{B})_{ki}=(-1)^{k+i+1} M_{ik}(\wt{B})\\
&=&(-1)^{k+i+1}(-1) ^{j-i-1} M_{jk}(B)\\
&=&(-1)^{k+j} M_{jk}(B)=(\Adj B)_{kj},
\end{array}
\]
as claimed.
\end{proof}

\medskip
There is a ``row version'' of Lemma \ref{lem.2} which can be proved along the same lines.
\begin{lemma}\label{lem.3}
Let $A\in\Rc^{n\times n}$ be a matrix  such that $\rm{rk}(A)<n$ and let $v\in\Rc^{n\times 1}$
be a right eigenvector of $A$ for the eigenvalue $0$. 
For $j=1,\ldots, n$ let $(\Adj A)^j$ be the $j$th row
of $\Adj(A)$.  Then, for all $i,j=1,\ldots,n$,
\be\label{eq.7}
v_i(\Adj A)^j=v_j(\Adj A)^i,
\ee
where $v=\begin{bmatrix} v_1& v_2 &\cdots & v_n\end{bmatrix}^T$.
\end{lemma}

The proof of Theorem \ref{thm.1} which follows is very much that of Grinberg in \cite{Gr19}.
It is included for completion and reader's convenience.

\begin{proof}[Proof of  Theorem \ref{thm.1}]
Let $B=\la_0 I_n-A$ and $p_B(\la)=\det (\la I_n-B)$ its
characteristic polynomial. Then $p_B(\la)=\la^n+\sum\limits_{k=1}^{n} (-1)^k c_k \la^{n-k}$ where,
for $k=0,\ldots, n$, $c_k=\sum\limits_{1\leq i_1 <\cdots <i_k\leq n} \det B(i_1: i_k,i_1: i_k)$, 
and $B(i_1: i_k,i_1: i_k)=\begin{bmatrix} b_{i_j,i_\ell}\end{bmatrix}_{1\leq j,\ell\leq k}$ is
the principal submatrix of $B$ formed by the rows and columns $i_1$, \ldots, $i_k$. In particular,
$c_{n-1}=\sum\limits_{j=1}^nM_{jj}(B)$ where $M_{jj}(B)$ is the principal minor of $B$ obtained by
deleting the $j$th row and column. Thus $p^\prime_B(0)=(-1)^{n-1} \sum\limits_{j=1}^nM_{jj}(B)$.

On the other hand, $\det (\la I_n-B)=\det (\la I_n-\la_0 I_n +A)=(-1)^n  \det ((\la_0-\la)I_n -A)=
(-1)^n p_A(\la_0-\la)$. It follows from the definition of derivative of a polynomial that
\[
p^{\prime}_A(\la_0)=(-1)^{n+1} p^{\prime}_B(0)=\sum\limits_{j=1}^nM_{jj}(\la_0 I_n-A).
\]
Hence, proving \eqref{eq.TM1} is equivalent to proving
\be\label{eq.8}
w^Tv\adj(B)=\sum_{j=1}^n M_{jj}(B) vw^T
\ee
where $B=\la_0 I_n-A$. It follows from $Av=\la_0 v$ and $w^T A=\la_0w^T$ that
$Bv=0$ and $w^TB=0$, respectively. So we can apply to $B$ properties \eqref{eq.5}
and \eqref{eq.7}. It follows from \eqref{eq.5} that $w_k(\Adj B)_{ij}= w_j (\Adj B)_{ik}$
for all $i,j,k\in\{1,\ldots, n\}$. Then $v_k w_k(\Adj B)_{ij}= w_j v_k(\Adj B)_{ik}$ and
from \eqref{eq.7}, $v_k(\Adj B)_{ik}=v_i(\Adj B)_{kk}$. Hence,
\[
v_k w_k(\Adj B)_{ij}=v_iw_j (\Adj B)_{kk},\quad i,j,k=1,\ldots, n.
\]
Adding on $k$ and taking into account that $(\Adj B)_{kk}=M_{kk}(B)$, we get
\[
w^Tv(\Adj B)_{ij}=\sum_{k=1}^nM_{kk}(B)v_iw_j,\quad i,j=1,\ldots n.
\]
This is equivalent to \eqref{eq.8} and the theorem follows.
\end{proof}

\medskip
\begin{rem}\label{rem.1}{\rm
Assume that $\Rc$ is an integral domain  and note that in this case
$\rm{rk}(A)=\rank (A)$; i.e., the McCoy rank and the usual rank coincide. 
It is an interesting consequence of \eqref{eq.TM1}
that $w^Tv=0$ implies $p'_A(\la_0)=0$. The converse is 
not true in general. For example, if $A=\la_0I_2$ then $v=\begin{bmatrix} 1 &0\end{bmatrix}^T$
satisfies both  $Av=\la_0 v$ and $v^TA=\la_0v^T$, but $v^Tv=1$
and $p'_A(\la_0)=0$. However, if $p'_A(\la_0)=0$ and $\rank (\la_0 I_n-A)=n-1$  then, necessarily,
$w^Tv=0$ because $\Adj(\la_0 I_n-A)$ is not the zero matrix. 
In particular, if $\F$ is a field of characteristic zero (see \cite[Ch. 30]{Gode05})
then it follows from \eqref{eq.TM1} that if $w^Tv=0$ then $\la_0$ is an eigenvalue
of algebraic multiplicity at least $2$. On the other hand, it is easily checked that if $\la_0$ is an eigenvalue
of algebraic multiplicity bigger that $1$ and geometric multiplicity $1$ then $w^Tv=0$ for any right and left
eigenvectors, $v$ and $w$ respectively,  of $A$ for $\la_0$. This is the case, for example,
of $A=\begin{bsmallmatrix} 0 & 0\\1 & 0 \end{bsmallmatrix}$. For this matrix, the TM formula
\eqref{eq.TM1} does not provide any substantial information about $\Adj(\la_0I_n-A)$ because,
in this case, $w^Tv=0$ and $p_A^\prime(\la_0)=0$. Thus, the TM formula \eqref{eq.TM1} is 
relevant for matrices with simple eigenvalues.
}\hfill $\Box$
\end{rem}

Our next goal is to provide a generalization of the TM formula \eqref{eq.TM1}
which is meaningful for nondiagonalizable matrices over fields. 
We will use the following notation: $\FF$ will denote an arbitrary field. If $A\in\FF^{n\times n}$
then $p_1(\la)$,\ldots, $p_r(\la)$ will be its (possibly repeated) elementary divisors in $\FF$
(\cite[Ch. VI, Sec. 3]{Gant88}). These are powers of monic irreducible polynomials of $\FF[\la]$
(the ring of polynomials with coefficients in $\F$). We will assume that for $j=1,\ldots,r$,
\[
p_j(\la)= \la^{d_j}+a_{j1}\la^{d_j-1}+ a_{j2}\la^{d_j-2}+\cdots+ a_{jd_j-1}\la+a_{jd_j}.
\]
Let $\Delta_A$ denote the determinant of $A$ and $\Lambda(A)$ the set of eigenvalues (the spectrum)
of $A$ in, perhaps, an extension field, $\wt{\FF}$, of $\FF$. Thus $\la_0\in\Lambda(A)$ if and only if
it is a root in $\wt{\FF}$ of $p_j(\la)$ for some $j\in\{1,2,\ldots, r\}$. In particular, 
$p_A(\la)=\prod_{j=1}^r p_j(\la)$ is the characteristic polynomial of $A$. 

Item (ii) of  the following theorem is an elementary result that 
is included for completion.

\begin{theorem}\label{thm.main}
With the above notation:
\begin{itemize}
\item[(i)] If $~0\not\in\Lambda(A)$ then the elementary divisors of $\Adj(A)$ are $q_1(\la)$,\ldots,
$q_r(\la)$ where for $j=1,\ldots, r$,
\begin{equation}\label{eq.elediv}
q_j(\la)=\la^{d_j}+\Delta_A\frac{a_{jd_j-1}}{a_{jd_j}}\la^{d_j-1}+\cdots+\Delta_A^{d_j-1}\frac{a_{j1}}{a_{jd_j}}\la+\Delta_A^{d_j}\frac{1}{a_{jd_j}}.
\end{equation}
\item[(ii)] If $~0\in\Lambda(A)$ and there are two indices $i,k\in\{1,\ldots, r\}$, $i\neq k$,
such that $p_i(0)=p_k(0)=0$ then $\Adj(A)=0$.
\item[(iii)] If $~0\in\Lambda(A)$, $p_k(0)=0$ for only one value $k\in\{1,\ldots, r\}$ and
$u,v\in\FF^{n\times 1}$ are arbitrary right and left eigenvectors of $A$,
respectively, for the eigenvalue $0$, then $v^TA^{d_k-1}u\neq 0$ and
\begin{equation}\label{eq.adja}
\Adj(A)=\frac{(-1)^{n-1}}{d_k!}p_A^{(d_k)}(0)\frac{uv^T}{v^TA^{d_k-1}u},
\end{equation}
where $p_A^{(d_k)}(\la)$ is the $d_k$-th derivative of $p_A(\la)$.
\end{itemize}
\end{theorem}

\begin{proof} For $j=1,\ldots, r$, let the companion matrix of $p_j(\la)$ be
\begin{equation}\label{eq.companion}
C_j=\begin{bmatrix}
0 & 0 & \cdots & 0 & -a_{jd_j}\\
1 & 0 &  \cdots & 0 & -a_{jd_j-1}\\
0&1 & \cdots & 0 & -a_{jd_j-2}\\
\vdots& \vdots &\ddots& \vdots&\vdots \\
0 & 0 &\cdots & 1&-a_{j1}
\end{bmatrix}.
\end{equation}
Then (see  \cite[Ch. VI, Sec. 6]{Gant88}) there is an invertible matrix $S\in\FF^{n\times n}$ such that
\begin{equation}\label{eq.adjc}
C=S^{-1}A S=\bigoplus_{j=1}^r C_j.
\end{equation}
An explicit computation shows that
\[
\Adj(C_j)= (-1)^{d_j}\begin{bmatrix}
-a_{jd_j-1}& a_{j d_j} & 0 & \cdots & 0\\
-a_{jd_j-2} & 0&a_{j d_j}  & \cdots & 0\\
\vdots & \vdots & \vdots&\ddots &\vdots\\
-a_{k1}& 0 & 0& \cdots  &a_{j d_j} \\
-1 & 0 & 0 & \cdots  &0
\end{bmatrix}.
\]
Bearing in mind that $\det C_j=(-1)^{d_j}a_{jd_j}$, we obtain $\Adj(C)=\oplus_{j=1}^r L_j$ where, for
$j=1,\ldots, r$,
\begin{equation}\label{eq.Lj}
L_j=(-1)^n \prod_{i=1,i\neq j}^r a_{id_i}\begin{bmatrix}
-a_{jd_j-1}& a_{j d_j} & 0 & \cdots & 0\\
-a_{jd_j-2} & 0&a_{j d_j}  & \cdots & 0\\
\vdots & \vdots & \vdots&\ddots &\vdots\\
-a_{k1}& 0 & 0& \cdots  &a_{j d_j} \\
-1 & 0 & 0 & \cdots  &0
\end{bmatrix}.
\end{equation}
Therefore, from \eqref{eq.adjc} we get
\begin{equation}\label{eq.adjads}
\Adj(A)=S\left(\bigoplus_{j=1}^r L_j\right)S^{-1}.
\end{equation}
\begin{itemize}
\item[(i)] Assume that $0\not\in\Lambda(A)$. This means that $a_{jd_j}\neq 0$ for all $j=1,\ldots, r$
and we can write
\[
L_j= \det A\begin{bmatrix}
-\frac{a_{jd_j-1}}{a_{jd_j}}& 1& 0 & \cdots & 0\\
-\frac{a_{jd_j-2}}{a_{j d_j}} & 0&1 & \cdots & 0\\
\vdots & \vdots & \vdots&\ddots &\vdots\\
-\frac{a_{j1}}{a_{j d_j}}& 0 & 0& \cdots  &1\\
-\frac{1}{a_{j d_j}} & 0 & 0 & \cdots  &0
\end{bmatrix}.
\]
Taking into account the definition of $q_j(\la)$ of \eqref{eq.elediv},
\begin{align*}
&\det(\la I_{d_j}-L_j)\\
&=\Delta_A^{d_j}\left(\frac{\la^{d_j}}{\Delta_A^{d_j}} +\frac {a_{jd_j-1}}{a_{jd_j}}
\frac{\la^{d_j-1}}{\Delta_A^{d_j-1}}+\cdots+ \frac{a_{j1}}{a_{j d_j}}\frac{\la}{\Delta_A}+
\frac{1}{a_{j d_j}}\right)=q_j(\la).
\end{align*}
Let us see that $q_j(\la)$ is a power of an irreducible polynomial in $\FF[\la]$.  In fact, put
\[
s_j(\la)=\la^{d_j}p_j\left(\frac{1}{\la}\right)=
a_{j d_j}\la^{d_j}+a_{jd_j-1}\la^{d_j-1}+\cdots+ a_{j1}\la+1.
\]
This polynomial is sometimes called the \textit{reversal polynomial} of $p_j(\la)$ (see, for example,
 \cite{MMMM06}). Since $p_j(\la)$ is an elementary divisor of $A$ in $\FF$, it is a power of an irreducible 
polynomial of $\FF[\la]$. By \cite[Lemma 4.4]{AMZ13}, $s_j(\la)$ is also
a power of an irreducible polynomial. Now, it is not difficult to see that 
$q_j(\la)=\frac{1}{a_{jd_j}}s\left( \frac{\la}{\Delta_A}\right)$ is  a power of
an irreducible polynomial too. As a consequence, $q_1(\la), q_2(\la), \ldots, q_r(\la)$ are
the elementary divisors of $\Adj(C)=\oplus_{j=1}^r L_j$. Since this and $\Adj(A)$
are similar matrices (cf. \eqref{eq.adjads}), $q_1(\la)$, $q_2(\la), \ldots, q_r(\la)$ are
the elementary divisors of $\Adj(A)$. This proves (i).

\item[(ii)] If $p_i(0)=p_j(0)=0$ for $i\neq j$, then $\rank(A)=\rank(C)\leq n-2$. Hence  all minors of
$A$ of order $n-1$ are equal to zero and so  $\Adj(A)=0$.

\item[(iii)] Assume now that there is only one index $k\in\{1,\ldots, r\}$ such that $a_{kd_k}=0$.
Then $p_k(\la)= \la^{d_k}$ because it is a power of an irreducible polynomial. Thus $a_{kj}=0$
for $j=1,\ldots, d_k$ and by \eqref{eq.companion} and \eqref{eq.Lj}, 
$C_k=\begin{bsmallmatrix} 0 & 0\\I_{d_k-1} & 0\end{bsmallmatrix}$ and
\begin{equation}\label{eq.Lk}
\begin{array}{rcl}
L_k&=&(-1)^{n-1} \prod\limits_{j=1,j\neq k}^ra_{jd_j}\begin{bmatrix}0\\0\\\vdots\\0\\1\end{bmatrix}
\begin{bmatrix}1&0&\cdots &0 &0\end{bmatrix}\\
&=&(-1)^{n-1} \prod\limits_{j=1,j\neq k}^ra_{jd_j} e_{d_k}e_1^T,
\end{array}
\end{equation}
respectively. Also, it follows from $a_{kd_k}=0$ that $L_j=0$ for $j=1,\ldots, r$, $j\neq k$.

Recall now that $S^{-1}A S= C=\oplus_{j=1}^r C_j$ and split $S$ and $S^{-1}$ accordingly:
\[
S=\begin{bmatrix} S_1 & S_2 & \cdots & S_r\end{bmatrix},\quad S^{-1}=\begin{bmatrix}
T_1\\T_2\\\vdots\\T_r\end{bmatrix},
\]
with $S_j\in\FF^{n\times d_j}$ and $T_j\in\FF^{d_j\times n}$, $j=1,\ldots, r$. Then
\begin{equation}\label{eq.assctact}
AS_k=S_kC_k,\quad T_kA=C_kT_k.
\end{equation}
For $i=1,\ldots, d_k$ let $s_{ki}$ and $t_{ki}^T$ be the $i$-th column and row of $S_k$ and $T_k$,
respectively:
\[
S_k=\begin{bmatrix}s_{k1}&s_{k2}& \cdots & s_{kd_k}\end{bmatrix},
\quad T_k=\begin{bmatrix} t_{k1}^T\\t_{k2}^T\\\vdots \\t_{kd_k}^T\end{bmatrix}.
\]
Bearing in mind that $\Adj(A)=S(\oplus_{j=1}^r L_j)S^{-1}$ (cf. \eqref{eq.adjads}), 
the representation of $L_k$ as a rank-one matrix of \eqref{eq.Lk} and that $L_j=0$
for $j\neq k$, we get
\begin{equation}\label{eq.rank1A1}
\Adj(A)=S_kL_kT_k=(-1)^{n-1}\left(\prod_{j=1,j\neq k}^r a_{jd_j}\right) s_{kd_k}t_{k1}^T.
\end{equation}
Now, it follows from \eqref{eq.assctact} that
\[
\begin{array}{lll}
s_{kj}=A s_{kj-1}, &\quad t_{kj-1}^T=t_{kj}^T A, &\quad j=2,3,\ldots, d_k,\\
As_{kd_k}=0, & \quad t_{k1}^T A=0. &
\end{array}
\]
Henceforth, $s_{kd_k}$ and $t_{k1}^T$ are right and left eigenvectors of $A$
for the eigenvalue $0$. Also, $\FI_k=<s_{k1}, A s_{k1},\ldots, A^{d_{k-1}}s_{k1}>$ is a
\textit{cyclic} $A$-invariant subspace with $s_{k1}$ as generating vector. Similarly,
$\FJ_k=<t_{kd_k}, A^Tt_{kd_k},\ldots, $ $(A^T)^{d_{k-1}}t_{kd_k}>$ is a \textit{cyclic}
$A^T$-invariant subspace with $t_{kd_k}$ as generating vector. Thus \eqref{eq.rank1A1}
is an explicit rank-one representation of $\Adj(A)$ in terms of a right and a left
eigenvectors of $A$ for the eigenvalue zero. Actually this representation depends on
a particular normalization of the vectors which span the cyclic subspaces $\FI_k$ and $\FJ_k$.
Specifically,  $T_kS_k=I_{d_k}$. However, we are looking for a more
general representation in terms of arbitrary right and left eigenvectors for which such a
normalization may fail to hold.

Let us assume that $u,v\in\FF^{n\times 1}$ are arbitrary right
and left eigenvectors of $A$ for the eigenvalue $0$. Then $Au=0$ and $v^TA=0$ and
since $\ker A=\ker A^T=1$, there are nonzero scalars $\alpha_1,\beta_1\in\FF$ such that
$u=\alpha_1s_{kd_k}$ and $v=\beta_ 1 t_{k1}$.  Put $u_{d_k}=u$, $v_1=v$ and for
$j=1,2,\ldots, d_k-1$ define
\[
\begin{array}{lcl}
u_{d_k-j}&=&\alpha_{j+1}s_{kd_k}+\alpha_j s_{kd_k-1}+\ldots+\alpha_1 s_{kd_k-j}\\
v_{j+1}&=&\beta_{j+1}t_{k1}+\beta_j t_{k2}+\ldots+\beta_1 t_{kj+1}
\end{array}
\]
with $\alpha_2,\ldots,\alpha_{d_k},\beta_2,\ldots,\beta_{d_k}\in\FF$ arbitrary scalars.
Using these scalars we define the following triangular matrices
\[
X=\begin{bsmallmatrix}
\alpha_1 & & &&\\
\alpha_2 & \alpha_1 & && \\
\vdots & \vdots & \ddots &&\\
\alpha_{d_k-1} & \alpha_{d_k-2} & \cdots& \alpha_1 &\\
\alpha_{d_k}& \alpha_{d_k-1} & \cdots & \alpha_2& \alpha_1 &
\end{bsmallmatrix}, \quad Y=\begin{bsmallmatrix}
\beta_1 & \beta_ 2 & \cdots &\beta_{d_k-1} & \beta_{d_k}\\
 & \beta_1 & \cdots & \beta_{d_k-2} & \beta_{d_k-1}\\
  & & \ddots &\vdots & \vdots\\
  &&&\beta_1 &\beta_2 \\
  &&&&\beta_1\end{bsmallmatrix}
\]
It is plain that $\begin{bmatrix} u_1 & u_2 & \cdots &u_{d_k}\end{bmatrix}=
\begin{bmatrix} s_{k1} & s_{k2} & \cdots &s_{kd_k}\end{bmatrix}X$ and also
$\begin{bmatrix} v_1 &v_2 & \cdots &v_{d_k}\end{bmatrix}=
\begin{bmatrix} t_{k1} & t_{k2} & \cdots &t_{kd_k}\end{bmatrix}Y$.
Since $X$ and $Y$ are nonsingular matrices for any choice of
$\alpha_2,\ldots,\alpha_{d_k},\beta_2,\ldots,\beta_{d_k}$
(because $\alpha_1\neq 0$ and $\beta_1\neq 0$), we conclude that
$\FI_k=<u_1,u_2,\ldots, u_{d_k}>$ and $\FJ_k=<v_1,v_2,\ldots, v_{d_k}>$.
In addition, for $j=1,2,\ldots, d_k-1$
\[
\begin{array}{rcl}
Au_{d_k-j}&=&\alpha_{j+1}As_{kd_k}+\alpha_j As_{kd_k-1}+\cdots+\alpha_1 As_{kd_k-j}\\
&=& \alpha_j s_{kd_k}+\alpha_{j-1}s_{kd_k-1}+\cdots+\alpha_1 As_{kd_k-j+1}=u_{d_k-j+1},
\end{array}
\]
and
\[
\begin{array}{rcl}
v_j^TA&=&\beta_jt_{k1}^T A+\beta_{j-1}t_{k2}^T A+\cdots+ \beta_1 t_{kj}^T A\\
&=&\beta_{j-1}t_{k1}^T+\beta_{j-2}t_{k2}^T+\cdots+ \beta_1t_{kj-1}=v_{j-1}^T.
\end{array}
\]
In other words, $u_1$ and $v_{d_k}$ are generating vectors of $\FI_k$ and $\FJ_k$
and $u=u_{d_k}=A^{d_k-1}u_1$ and $v=v_1=A^Tv_{d_k}$ are the given
right and left eigenvectors of $A$ for the eigenvalue $0$. Now, it follows from
$u=\alpha_1 s_{kd_k}$, $v=\beta_1 t_{1k}$  and
\eqref{eq.rank1A1} that
\begin{equation}\label{eq.rank1A2}
\Adj(A)= (-1)^{n-1}\left(\prod_{j=1,j\neq k}^r a_{jd_j}\right)  \frac{u v^T}{\alpha_1\beta_1}.
\end{equation}
Since $T_kS_k=I_{d_k}$,
\[
\begin{bmatrix} v_1^T\\v_2^T\\\vdots\\v_{d_k}^T\end{bmatrix}
\begin{bmatrix} u_1 & u_2 & \cdots &u_{d_k}\end{bmatrix}= Y^TT_kS_kX=Y^TX.
\]
But $Y^TX$ is a lower triangular matrix whose diagonal elements are all equal to
$\alpha_1\beta_1$. Thus, for $j=1,\ldots, d_k$,
$\alpha_1\beta_1=v_j^Tu_j=v_1^T A^{d_k-1}u_{d_k}=v^TA^{d_k-1}u$. Since $\alpha_1\neq 0$
and $\beta_1\neq 0$, $v^TA^{d_k-1}u\neq 0$ as claimed. Now, from \eqref{eq.rank1A2}
\begin{equation}\label{eq.rank1A3}
\Adj(A)= (-1)^{n-1}\left(\prod_{j=1,j\neq k}^r a_{jd_j}\right)  \frac{u v^T}{v^TA^{d_k-1}u}.
\end{equation}
Finally, $p_A(\la)=\prod\limits_{j=1}^r p_j(\la)=\la^{d_k}\prod\limits_{j=1,j\neq k}^r p_j(\la)$.
Therefore, $p_A^{(d_k)}(0)=d_k!\prod\limits_{j=1,j\neq k}^r p_j(0)=
d_k!\prod\limits_{j=1,j\neq k}^r a_{jd_j}$
and \eqref{eq.adja} follows from \eqref{eq.rank1A3}. 
\end{itemize}
\end{proof}

As a first consequence of Theorem \ref{thm.main} we present a generalization of
the formula for the eigenvalues of the adjugate matrix (see \cite{HiUn85}).

\begin{corollary}\label{coro.eigval}
Let $A\in\FF^{n\times n}$ be a nonsingular matrix. Let $\la_0\in\Lambda(A)$ and let
$m_1\geq \ldots\geq m_s$ be its partial multiplicities (i.e., the sizes of the Jordan blocks 
associated to $\la_0$ in any Jordan form of $A$ in, perhaps, a extension field $\,\wt{\FF}$.
Then $\frac{\Delta_A}{\la_0}$ is an eigenvalue of $\Adj(A)$ with 
$m_1\geq \ldots\geq m_s$ as partial multiplicities.
\end{corollary}

\begin{proof}The elementary divisors of $A$ for the eigenvalue $\la_0$ in  $\wt{\FF}(\la)$ are
$(\la-\la_0)^{m_1}$, \ldots, $(\la-\la_0)^{m_s}$. Then, it follows from  item (i) of
Theorem \ref{thm.main} (see \eqref{eq.elediv}) that $\left(\la-\frac{\Delta_A}{\la_0}\right)^{m_1},
\ldots, \left(\la-\frac{\Delta_A}{\la_0}\right)^{m_s}$ are the corresponding elementary
divisors of $\Adj(A)$.
\end{proof}

\begin{corollary}\label{coro.TMg}
Let $A\in\FF^{n\times n}$, $\la_0\in\Lambda(A)\cap\FF$ and let $m_1\geq \ldots\geq m_s$ be
its partial multiplicities. Let $u,v\in\FF^{n\times 1}$ be arbitrary right
and left eigenvectors of $A$ for $\la_0$. Then
\begin{equation}\label{eq.TMg}
\Adj(\la_0 I_n- A)=\frac{(-\delta_{1s})^{m-1}}{m!}\;p^{(m)}_A(\la_0)\;\frac{uv^T}{v^T(\la_0I_n-A)^{m-1}u}
\end{equation}
where $m$ is the algebraic multiplicity of $\la_0$ and $\delta_{lj}$ is the Kronecker delta.
\end{corollary}

\begin{proof}Put $B=\la_0 I_n-A$. Then $0\in\Lambda(B)$, $u$ and $v$
are right and left eigenvectors of $B$ for the eigenvalue $0$ and
$m_1\geq\ldots\geq m_s$ are the partial multiplicities of this eigenvalue.
By Theorem \ref{thm.main}, $\Adj(\la_0 I_n-A)=\Adj(B)\neq 0$ if and only if $s=1$.
In this case,
\[
\Adj(\la_0 I_n- A)=\Adj(B)=\frac{(-1)^{n-1}}{m!}p_B^{(m)}(0)\;\frac{uv^T}{v^TB^{m-1}u}.
\]
Therefore \eqref{eq.TMg} follows from the fact that $p^{(m)}_B(0)=(-1)^{n+m}p^{(m)}_A(\la_0)$
(see the proof of Theorem \ref{thm.1}).
\end{proof}

The following result is an immediate consequence of Corollary \ref{coro.TMg}.

\begin{corollary}\label{coro.TMgall}
Let ${A}\in\FF^{n\times n}$ and let $\Lambda(A)=\{\la_1,\ldots,\la_s\}$ be its spectrum.
Assume that $\Lambda(A)\subset\FF$ and let $m_j$ and $g_j$ be the algebraic and
geometric multiplicities of $A$ for the eigenvalue $\la_j$, $j=1,\ldots, s$.
Fix $k\in \{1,\dots, s\}$ and let $u_k$ and $v_k$ be right and
left eigenvectors of ${A}$ for $\la_k$. Then 
\begin{align}\label{eq..TMgall}
\Adj(\la_k\, {I}-{A})=(-\delta_{1g_k})^{m_k-1}{\prod_{j=1,\, j\neq k}^s}  (\la_k-\la_j)^{m_j}
\frac{u_kv^T_k}{v^T_k A^{m_k-1} u_k}.
\end{align}
\end{corollary}

The TM formula \eqref{eq.TM1} can be used to provide an easy proof of the so-called 
eigenvector-eigenvalue identity (see \cite[Sec. 2.1]{BPTZ20}). In fact,
under the hypothesis of Theorem \ref{thm.1},  it follows from \eqref{eq.TM1} that
$w^T v[\adj(\la_0 I_n -A)]_{jj}= p'_A(\la_0)v_jw_j$, $j=1,\ldots, n$ (see \cite[Rem 5]{BPTZ20}).
Hence, recalling that for $j=1,\ldots, n$, $M_{jj}$ is the principal minor of $\la_0 I_n -A$
obtained by removing its $j$th row and column,
\be\label{eq.ev-ev1}
(w^T v)\;p_{M_{jj}}(\la_0)= p'_A(\la_0)\; v_jw_j,\quad j=1,\ldots, n.
\ee
In particular, if $A\in\CC^{n\times n}$ is Hermitian, $\la_1\geq \la_2\geq\cdots\geq \la_n$ are
its eigenvalues and, for $i=1,\ldots, n$,  $v_i=\begin{bmatrix} v_{i1} & v_{i2} & \cdots & v_{in}\end{bmatrix}^T$
is a unitary right and left eigenvector of $A$ for $\la_i$ (that is $Av_i=\la_i v_i$,
$v_i^\ast A=\la_iv_i^\ast$ and $v_i^\ast v_i=1$; (recall that we must change
\textit{transpose} by \textit{conjugate transpose} in the complex case) then
\[
|v_{ij}|^2p'_A(\la_i)=p_{M_{jj}}(\la_i), \quad i,j=1,\ldots, n.
\]
Equivalently, if $\mu_{j1}\geq \mu_{j2}\geq \cdots\geq \mu_{jn-1}$ are the eigenvalues of
$M_{jj}$,
\be\label{eq.ev-ev0}
|v_{ij}|^2\prod_{k=1,k\neq i}^n(\la_i-\la_k)=\prod_{k=1}^n (\la_i-\mu_{jk})\quad i,j=1,\ldots, n.
\ee
This is the classical eigenvector-eigenvalue identity (see \cite[Thm. 1]{BPTZ20}). 

As mentioned in Remark \ref{rem.1}, if $\FF$ is a field of characteristic zero and
$A\in\FF^{n\times n}$ then \eqref{eq.ev-ev1}  is meaningful if and only if $\la_0$ is a simple
eigenvalue. If $\la_0$ is defective and its geometric multiplicity is bigger than $1$
then \eqref{eq.ev-ev1} becomes a trivial identity because, in this case, $\Adj(\la_0 I_n-A)=0$
(item (ii) of Theorem \ref{thm.main}) and so $p_{M_{j}}(\la_0)=\det (\la_0 I_{n-1}-M_{j})=0$.
However, if  $\la_0$ is defective and its geometric multiplicity is $1$, then \eqref{eq.TMg}
can be used to obtain a generalization of the eigenvector-eigenvalue identity.  In fact,
one readily gets from \eqref{eq.TMg}:
\be\label{eq.ev-evg}
p_{M_{jj}}(\la_0)=\frac{(-\delta_{1g})^{m-1}}{m!}
p_A^{(m)}(\la_0) \frac{u_j v_j}{v^T(\la_0 I_n-A)^{m-1}u},\quad j=1,\ldots, n,
\ee
where $m$ and $g$ are the algebraic  and geometric multiplicities of $\la_0$, respectively. Moreover,
if both $p_A(\la)$ and  $p_{M_{jj}}(\la)$ split in $\FF$ then, with the notation of Corollary \ref{coro.TMgall}, 
the following identity follows from \eqref{eq..TMgall} for the non-repeated eigenvalues
$\{\mu_{j1},\ldots, \mu_{jr_j}\}$ of $M_{jj}$ and for $i=1,\ldots, s$:
\be\label{eq.ev-evgall}
\prod_{k=1}^{r_k}(\la_i-\mu_{jk})^{q_{jk}}=(-\delta_{1g_i})^{m_i-1}
\frac{u_{ij}v_{ij}}{v_i^T A^{m_i-1}u_i}\prod_{k=1,k\neq i}^s (\la_i-\la_k)^{m_k},
\quad j=1,\ldots, n,
\ee
where $u_i=\begin{bmatrix} u_{i1} & \cdots & u_{in}\end{bmatrix}^T$, 
$v_i=\begin{bmatrix} v_{i1} & \cdots & v_{in}\end{bmatrix}^T$, 
and $q_{jk}$ is the algebraic multiplicity of $\mu_{jk}$, $k=1,\ldots, r_j$ and $j=1,\ldots, n$.

\medskip
In the following section two additional applications will be presented.

\section{Two additional consequences of the TM formula}\label{sec.cons}

The well-known formula \eqref{derivative} below gives the derivative of a simple eigenvalue
of a matrix depending on a (real or complex) parameter. The investigation about the eigenvalue
sensitivity of matrices depending on one or several parameters can be traced back to the work
of Jacobi (\cite{Jacobi}). However a systematic study of the perturbation theory of the
eigenvalue problem starts with the books of Rellich (1953), Wilkinson (1965) and Kato (1966), 
as well as the papers by Lancaster \cite{Lan64}, Osborne and Michaelson \cite{OsMi64},
Fox and Kapoor \cite{FoKa68}, Crossley and Porter \cite{CroPo69} (see also \cite{Sun90} and
the references therein).
Since then this topic has become classical
as evidenced by an extensive literature including books and papers addressed to mathematicians
and a broad spectrum of scientist and engineers. In addition to the above early references, a short, 
and by no means exhaustive, list of books could include \cite[p. 463]{A64},
\cite[Ch. 8, Sec. 9]{MagNeud88}, \cite[Sec.4.2]{HinPrit05} or \cite[pp. 134-135]{L07}.

In proving \eqref{derivative}, one first must prove, of course,
that the eigenvalues smoothly depend on the parameter. It is also a common practice to prove
or assume (see \cite{Mag85},\cite[Ch. 11, Th. 2]{E10}  and the referred books), the existence
of eigenvectors which depend smoothly on the parameter. It is worth-remarking that  in the
proof by Lancaster in \cite{Lan64} only the existence of eigenvectors continuously depending
on the parameter is required. We propose a simple and alternative proof  of \eqref{derivative} 
where no assumption is made on the right and left eigenvector functions.

Let $D_\epsilon(z_0)$ be the open disc of radius $\epsilon>0$ with center $z_0$.
For the following result $\FF$ will be either the field of real numbers $\RR$
or of the complex numbers $\CC$.
Recall that $v\in\CC^{n\times 1}$ is a left eigenvector of $A\in\CC^{n\times n}$ for an eigenvalue
$z_0$ if $v^\ast A=z_0v^\ast$ where $v^\ast=\bar{v}^T$ is the transpose conjugate of $v$. 
Hence, we will change ${\rm }^T$ by ${\rm }^\ast$ to include complex vectors in our discussion.

\begin{proposition}\label{P1}
Let ${A}(\omega)\in\F^{n\times n}$ be a square matrix-valued function whose entries are analytic
at $\omega_0 \in \mathbb{C}$.  Let $z_0$ be a simple eigenvalue of ${A}(\omega_0)$. Then there exist
$\epsilon>0$ and $\delta>0$ so that $z : D_\epsilon(\omega_0) \to D_\delta(z_0)$ is the unique
eigenvalue of ${A}(\omega)$ with $z(\omega)\in D_\delta(z_0)$ for each $\omega \in D_\epsilon(\omega_0)$.
Moreover, $z$ is analytic on $D_\epsilon(\omega_0)$ and
\begin{align}\label{derivative}
z'(\omega)=\frac{v(\omega)^\ast {A}'(\omega)u(\omega)}{v(\omega)^\ast u(\omega)},
\end{align}
where, for $w\in D_\epsilon(\omega_0)$, $u(\omega)$ and $v(\omega)$ are arbitrary right
and left eigenvector, respectively, of ${A}$ for $z(\omega)$.
\end{proposition}
\begin{proof}
Since $z_0$ is a simple root of $p(z, \omega)=\det(z\, {I}-{A}(\omega))$, by the analytic implicit function theorem, 
we have, in addition to the first part of the result, that
$$
z'(\omega)=\dps -\frac{\dps \frac{\partial p}{\partial \omega}(z(\omega), \omega)}
{\dps \frac{\partial p}{\partial z} (z(\omega), \omega)}.
 $$
By the Jacobi formula for the derivative of the determinant and TM formula \eqref{eq.TM1}, we have
(note that since $z(\omega)$ is a simple eigenvalue, $v(\omega)^\ast u(\omega)\neq 0$ for any
right and left eigenvectors $u(\omega)$ and $v(\omega)$)
\begin{align*}
\dps \frac{\partial p}{\partial z}(z(\omega), \omega)&=\tr(\adj(z(\omega)\, {I}-{A}(\omega))\\
&=p'(z(\omega), \omega)\\
\dps \frac{\partial p}{\partial \omega}(z(\omega), \omega)&=-\tr(\Adj(z(\omega)\, {I}-{A}(\omega)) {A}'(\omega))\\
&=-p'(z(\omega), \omega)\frac{v(\omega)^\ast {A}'(\omega)u(\omega)}{v(\omega)^\ast u(\omega)},
\end{align*}
and the result follows.
\end{proof}

\begin{rem}{\rm
\begin{itemize}
\item[(a)]The same conclusion can be drawn in Proposition \ref{P1} if $A$ is a complex or real
matrix-valued differentiable  function of a real variable. In the first case, we would need
a non-standard version of the implicit function theorem like the one in 
\cite[Theorem 2.4]{AH82}. In the second case the standard  implicit function theorem is enough.
\item[(b)]  It is shown in \cite{AnChLa93} that the existence of eigenvectors smoothly depending on the
parameter can be easily obtained from the properties of the adjugate matrix. In fact,
since $z(\omega)$ is a simple eigenvalue of $A(\omega)$ for each $\omega\in D_\epsilon(\omega_0)$, 
$\rank (z(\omega) I_n- A(\omega))=n-1$ and so by the TM formula,
$\rank \Adj (z(\omega) I_n- A(\omega))=1$ (see Remark \ref{rem.1}). 
Now, $\Adj (z(\omega) I_n- A(\omega))$ is a
differentiable matrix function of $\omega\in D_\epsilon(\omega_0)$ and 
$(z(\omega) I_n- A(\omega))(\Adj (z(\omega) I_n- A(\omega)))=
(\Adj (z(\omega) I_n- A(\omega))) (z(\omega) I_n- A(\omega))=
\det(z(\omega) I_n- A(\omega))I_n=0$. Henceforth, all nonzero columns of
$\Adj (z(\omega) I_n- A(\omega)$, which are all proportional, are (right and left)
eigenvectors of $A(\omega)$ for $z(\omega)$.
\hfill$\Box$
\end{itemize}}
\end{rem}

The second application is related to the problem of characterizing the admissible eigenstructures
and, more generally, the similarity orbits of the rank-one updated matrices. There is a vast literature
on this problem. A non-exhaustive list of publications is  \cite{Thomp80,Silva88,Ion91,MoDo03,
BeHoZa05,MMRR11,BrCaUr15,Itziar20} and the references therein.  It is a consequence
of Theorem 2 in \cite{Thomp80} that if $\la_0$ is an eigenvalue of 
$A\in\FF^{n\times n}$ with geometric multiplicity $1$ and $\rank(B-A)=1$ then $\la_0$
may or may not be an eigenvalue of $B\in\FF^{n\times n}$. It is then proved in
\cite[Th. 2.3]{MMRR11} that in the complex case, generically, $\la_0$ is not an
eigenvalue of $B$. That is to say, there is a Zariski open set 
$\Omega\subset\CC^{n}\times \CC^n$ such that for all $(x,y)\in\Omega$, $\la_0$ is not
an eigenvalue of $A+xy^T$. With the help of the TM formula we can be a little more precise about
the set $\Omega$.  Form now on, $\FF$ will be again an arbitrary field.

\begin{proposition}\label{P2}
Let $A\in\F^{n\times n}$ and let $\la_0$ be an eigenvalue of $A$ in, perhaps, an extension
field $~\wt{\F}$. Assume that the geometric multiplicity of $\la_0$ is $1$ and its algebraic multiplicity
is $m$. Let $u_0,v_0\in\F^{n\times 1}$ be right and left eigenvectors of $A$ for $\la_0$. If
$x,y\in\F^{n\times 1}$ then $\la_0$ is an eigenvalue of $A+xy^T$ if and only if $y^Tu_0= 0$
or $v_0^Tx=0$.
\end{proposition}

\begin{proof}
Let $B=A+xy^T$. Then $\la I_n-A=\la I_n-B-xy^T$. Taking into account that
$\la I_n-B$ is invertible in $\F(s)^{n\times n}$, where $\F(s)$ the field of rational functions, and using
the formula of the determinant of  updated rank-one matrices, we get
\[
p_B(\la)=p_A(\la)+p_A(\la)y^T(\la I_n-A)^{-1}x=p_A(\la)+y^T \adj(\la I_n-A) x.
\]
In particular,
\begin{equation}\label{eq.detr1}
p_B(\la_0)=p_A(\la_0)+y^T \adj(\la_0 I_n-A) x= y^T \adj(\la_0 I_n-A) x.
\end{equation}
It follows from \eqref{eq.TMg} that (recall that $v_0 ^T(\la_0 I_n-A)^{m-1}u_0\neq 0$)
\[
p_B(\la_0)=\frac{(-1)^{m-1}}{m!}p_A^{(m)}(\la_0)\frac{y^Tu_0 v_0^T x}{v_0 ^T(\la_0 I_n-A)^{m-1}u_0}.
\]
Since $p_A^{(m)}(\la_0)\neq 0$, 
the Proposition follows.
\end{proof}

\begin{rem}\label{rem.3.4}{\rm
Note that, by \eqref{eq.detr1} and item (ii) of Theorem \ref{thm.main}, if the geometric multiplicity of
$\la_0$ as eigenvalue of $A$ is $2$ then $\Adj(\la_0 I_n-A)=0$ and so, $\la_0$ is necessarily an eigenvalue
of $A+xy^T$. This is an easy consequence of the interlacing inequalities of \cite[Th. 2]{Thomp80}. However,
proving that those interlacing inequalities are necessary conditions that the invariant polynomials of
$A$ and $A+xy^T$ must satisfy is by no means a trivial matter.
}\hfill$\Box$
\end{rem}

The eigenvalues of rank-one updated matrices are at the core of the
\textit{divide and conquer algorithm} to compute the eigenvalues of real symmetric
or complex hermitian matrices (see, for example,
\cite[Sec. 5.3.3]{Demm97}, \cite[Sec. 2.1]{Stew01}). At each step  of the algorithm a
diagonal matrix $D=D_1\oplus D_2$ and a vector $u\in\C^{n\times 1}$ are given such
that the eigenvalues and eigenvectors of $D+uu^\ast$ are to be computed. In order
the algorithm to run smoothly, it is required, among other things, that the diagonal
elements of $D$ are all distinct. Thus, a so-called \textit{deflation} process must be
carried out. This amounts to check at each step  the presence of repeated eigenvalues
and, if so, remove and save them. The result that follows is related to the problem of detecting
repeated eigenvalues but for much more general matrices over arbitrary fields.

\begin{proposition}\label{P3}
Let $A=A_1\oplus A_2$ with $A_i\in\F^{n_i\times n_i}$, $i=1,2$. Let $x,y\in\F^{n\times 1}$
and split $B=A+xy^T=\begin{bmatrix} B_{ij}\end{bmatrix}_{ij=1,2}$ into $2\times 2$ blocks
such that $B_{ii}\in\F^{n_i\times n_i}$, $i=1,2$.   Assume also that the eigenvalues of $A_1$
and $A_2$ have geometric multiplicity equal to  $1$ and
$\Lambda(A_1)\cap\Lambda(B_{11})= \Lambda(A_2)\cap\Lambda(B_{22})=\emptyset$.
Then
\[
\Lambda(A_1)\cap\Lambda(A_2)=\Lambda(B)\cap\Lambda(A_1)=\Lambda(B)\cap\Lambda(A_2).
\]
\end{proposition}

\begin{proof}.- If $\la_0\in\Lambda(A_1)\cap\Lambda(A_2)$ then $\la_0$, as eigenvalue of $A$,
has geometric multiplicity $2$. By Remark \ref{rem.3.4}, 
$\la_0\in\Lambda(B)\cap\Lambda(A_1)\cap\Lambda(A_2)$. Assume that$\la_0\in\Lambda(B)\cap\Lambda(A_1)$ but $\la_0\not\in\Lambda(A_2)$. Let us see that this assumption leads to a contradiction. Let $u_0,v_0\in\FF^{n_1\times 1}$
be a right and a left  eigenvectors of $A_1$, respectively. Then
$w_0=\begin{bmatrix} u_0 ^T & 0\end{bmatrix}^T\in\FF^{n\times 1}$
and $z_0= \begin{bmatrix} w_0 ^T & 0\end{bmatrix}^T\in\FF^{n\times 1}$are right and
left eigenvectors of $A$, respectively, for $\la_0$. Since $\la_0\not\in\Lambda(A_2)$,
the geometric multiplicity of $\la_0$ as eigenvalue of $A$ 
is $1$. Then, by Proposition \ref{P2},  $ y^Tw_0=0$ or $z_0^Tx=0$ because
$\la_0\in\Lambda(B)$.  Let us assume
that $ y^Tw_0=0$, on the contrary we would proceed similarly with $z_0^Tx=0$.
If we put $y=\begin{bmatrix} y_1^T & y_2^T\end{bmatrix}^T$ and
$x=\begin{bmatrix} x_1^T & x_2^T\end{bmatrix}^T$, with $x_1,y_1\in\FF^{n_1\times 1}$,
then $y_1^T u_0=0$ and $B_{11}=A_{11}+x_1y_1^T$. It follows from Proposition \ref{P2}
that $\la_0\in\Lambda(B_{11})$, contradicting the hypothesis
$\Lambda(A_1)\cap\Lambda(B_{11})=\emptyset$. That 
$\Lambda(B)\cap\Lambda(A_2)\subset\Lambda(A_1)\cap\Lambda(A_2)$ is proved similarly.
\end{proof}

\begin{rem}\label{rem.3.6}{\rm
\begin{itemize}
\item[{(i)}]  Note that, with the notation of the proof of Proposition \ref{P3}, $B_{11}=A_1+x_1y_1^T$
and $B_{22}=A_2+x_2y_2^T$. Then, according to Proposition \ref{P2}, 
$\la_0\not\in\Lambda(B_{11})$ unless $(y_1^T u_0)(v_0^Tx_1)=0$.
Hence, the hypothesis  $\Lambda(A_1)\cap\Lambda(B_{11})= \emptyset$
is a generic property, and so is $\Lambda(A_2)\cap\Lambda(B_{22})=\emptyset$.

\item[{(ii)}] Consider Proposition \ref{P3} over $\mathbb{C}$. If $A$ and $B$ are both Hermitian or unitary, 
then $\Lambda(B)\setminus \big(\Lambda(A_1) \cap \Lambda(A_2) \big)$  and
$\Lambda(A_1)\cup \big(\Lambda(A_2) \backslash (\Lambda(A_1)\cap \Lambda(A_2))\big)$
strictly interlace on the real line or the unit circle, respectively (see, for example, 
\cite[Th. 2.1, Sec. 2]{Stew01}).\hfill$\Box$
\end{itemize}
}
\end{rem}

\bigskip

\bigskip

\bibliographystyle{plain}
 \bibliography{bib}

\begin{thebibliography}{10}

\bibitem{AMZ13}
A.~Amparan, S.~Marcaida, and I.~Zaballa.
\newblock On the structure invariants of proper rational matrices with
  prescribed finite poles.
\newblock {\em Linear and Multilinear Algebra}, 61(11):1464--1486, 2013.

\bibitem{AnChLa93}
A.~L. Andrew, K.-W.~E. Chu, and P.~Lancaster.
\newblock Derivatives of eigenvalues and eigenvectors of matrix functions.
\newblock {\em SIAM J. Matrix Anal. Appl.}, 14(4):903--926, 1993.

\bibitem{AH82}
M.~S. Ashbaugh and E.~M.~Harrell {II}.
\newblock Perturbation theory for shape resonances and large barrier
  potentials.
\newblock {\em Comm. Math. Phys.}, 83(2):151--170, 1982.

\bibitem{A64}
F.~V. Atkinson.
\newblock {\em Discrete and continuous boundary problems}, volume~8 of {\em
  Mathematics in Science and Engineering}.
\newblock Academic Press, New York-London, 1964.

\bibitem{Itziar20}
I.~Baraga{\~n}a.
\newblock The number of distinct eigenvalues of a regular pencil and of a
  square matrix after rank perturbation.
\newblock {\em Linear Algebra Appl.}, 588:101--121, 2020.

\bibitem{BeHoZa05}
M.~A. Beitia, I.~de~Hoyos, and I.~Zaballa.
\newblock The change of the \uppercase{J}ordan structure under one row
  perturbations.
\newblock {\em Linear Algebra Appl.}, 401:119 -- 134, 2005.

\bibitem{Brown93}
W.~C. Brown.
\newblock {\em Matrices over Commutative Rings}.
\newblock Marcel Dekker Inc., New York, 1993.

\bibitem{BrCaUr15}
R.~Bru, R.~Cant{\'o}, and A.~M. Urbano.
\newblock Eigenstructure of rank one updated matrices.
\newblock {\em Linear Algebra Appl.}, 485:372--391, 2015.

\bibitem{CroPo69}
T.~R. Crossley and B.~Porter.
\newblock Eigenvalue and eigenvector sensitivities in linear system theory.
\newblock {\em Int. J. Control}, 10:163--170, 1969.

\bibitem{HinPrit05}
Hinrichsen D. and Pritchard~A. J.
\newblock {\em Mathematical System Theory I. Modelling, State Space Analysis,
  Stability and Robustness}.
\newblock Springer, Berlin, 2005.

\bibitem{Demm97}
J.~W. Demmel.
\newblock {\em Applied Numerical Linear Algebra}.
\newblock SIAM, Philadelphia, 1997.

\bibitem{BPTZ20}
P.~B. Denton, S.~J. Parke, T.~Tao, and X.~Zhang.
\newblock Eigenvectors from eigenvalues: a survey of a basic identity in linear
  algebra.
\newblock {\em arXiv:1908.03795}, 2020.

\bibitem{E10}
L.~C. Evans.
\newblock {\em Partial differential equations}, volume~19 of {\em Graduate
  Studies in Mathematics}.
\newblock American Mathematical Society, Providence, RI, second edition, 2010.

\bibitem{FoKa68}
R.~L. Fox and M.~P. Kapoor.
\newblock Rate of change of eigenvalues and eigenvectors.
\newblock {\em AIAA J.}, 6:2426--2429, 1968.

\bibitem{Gant88}
F.~R. Gantmacher.
\newblock {\em The Theory of Matrices}.
\newblock AMS Chelsea Publishing, Providence, Rhode Island, 1988.

\bibitem{Gode05}
R.~Godement.
\newblock {\em Cours d'alg\`ebre}.
\newblock Hermann \'Editeurs, Paris, 2005.

\bibitem{Gr19}
D.~Grinberg.
\newblock Eigenvectors from eigenvalues: a survey of a basic identity in linear
  algebra | what's new.
\newblock
  \url{https://terrytao.wordpress.com/2019/12/03/eigenvectors-from-eigenvalues-a-survey-of-a-basic-identity-in-linear-algebra/#comment-531597},
  2019.

\bibitem{HiUn85}
R.~D. Hill and E.~E. Underwood.
\newblock On the matrix adjoint (adjugate).
\newblock {\em SIAM J. Algebraic Discrete Methods}, 6(4):731--737, 1985.

\bibitem{Jacobi}
C.~G.~J. Jacobi.
\newblock \"{U}ber ein leichtes verfahren die in der theorie der
  s\"acularst\"orungen vorkommenden gleichungen numerisch aufzul\"osen.
\newblock {\em J. f\"ur die Reine und Angew. Math.}, 1846(30):51--94, 1846.

\bibitem{Lan64}
P.~Lancaster.
\newblock On eigenvalues of matrices dependent on a parameter.
\newblock {\em Numer. Math.}, 6:377--387, 1964.

\bibitem{L07}
P.~D. Lax.
\newblock {\em Linear Algebra and its Applications}.
\newblock Pure and Applied Mathematics (Hoboken). Wiley-Interscience [John
  Wiley \& Sons], Hoboken, NJ, second edition, 2007.

\bibitem{MMMM06}
D.~S. Mackey, N.~Mackey, C.~Mehl, and V.~Mehrmann.
\newblock Vector spaces of linearizations for matrix polynomials.
\newblock {\em SIAM J. Matrix Anal. Appl.}, 28(4):971--1004, 2006.

\bibitem{Mag85}
J.~R. Magnus.
\newblock On differentiating eigenvalues and eigenvectors.
\newblock {\em Econometric Theory}, 1:179--191, 1985.

\bibitem{MagNeud88}
J.~R. Magnus and H.~Neudecker.
\newblock {\em Matrix Differential Calculus with Applications in Statistics and
  Econometrics}.
\newblock John Wiley \& Sons, Chichester, 1988.

\bibitem{MMRR11}
C.~Mehl, V.~Mehrmann, A.~C.~M. Ran, and L.~Rodman.
\newblock Eigenvalue perturbation theory of classes of structured matrices
  under generic structured rank one perturbations.
\newblock {\em Linear Algebra Appl.}, 435(3):687--716, 2011.

\bibitem{MoDo03}
J.~Moro and F.~M. Dopico.
\newblock Low rank perturbation of \uppercase{J}ordan structure.
\newblock {\em SIAM J. Matrix Anal. Appl.}, 25(2):495--506, 2003.

\bibitem{OsMi64}
M.~R. Osborne and S.~Michaelson.
\newblock The numerical solution of eigenvalue problems in which the eigenvalue
  appears nonlinearly, with an application to differential equations.
\newblock {\em Computer J.}, 7:66--71, 1964.

\bibitem{S79}
D.~S. Scott.
\newblock How to make the {L}anczos algorithm converge slowly.
\newblock {\em Math. Comp.}, 33:239--247, 1979.

\bibitem{Silva88}
F.~C. Silva.
\newblock The rank of the difference of matrices with prescribed similarity
  classes.
\newblock {\em Linear and Multilinear Algebra}, 24(1):51--58, 1988.

\bibitem{Stew01}
G.~W. Stewart.
\newblock {\em Matrix Algorithms, Volume II: Eigensystems}.
\newblock SIAM, Philadelphia, 2001.

\bibitem{Sun90}
J.~G. Su.
\newblock Multiple eigenvalue sensitivity analysis.
\newblock {\em Linear Algebra Appl.}, 137(4):183--211, 1990.

\bibitem{Thomp80}
R.~C. Thompson.
\newblock Invariant factors under rank one perturbations.
\newblock {\em Canad. J. Math.}, 32(1):240--245, 1980.

\bibitem{TM68}
R.~C. Thompson and P.~Mc{E}nteggert.
\newblock Principal submatrices. {II}: {T}he upper and lower quadratic
  inequalities.
\newblock {\em Linear Algebra Appl.}, 1:211--243, 1968.

\bibitem{Ion91}
I.~Zaballa.
\newblock Pole assignment and additive perturbations of fixed rank.
\newblock {\em SIAM J. Matrix Anal. Appl.}, 12(1):16--23, 1991.

\end{thebibliography}
\end{document}